\documentclass[12pt]{article}
\usepackage[english]{babel}
\usepackage[tbtags]{amsmath}
\usepackage{amsfonts,amssymb,mathrsfs,amsmath,amscd,comment}

\usepackage{graphicx}
\graphicspath{{pictures/}} \DeclareGraphicsExtensions{.pdf,.png,.jpg}

\usepackage{amssymb}
\usepackage{amsthm}
\usepackage{amsmath}

\usepackage{todonotes}

\numberwithin{equation}{section}

\theoremstyle{plain}
\newtheorem{theorem}{Theorem}
\newtheorem*{theorem*}{Theorem}
\newtheorem{proposition}{Proposition}[section]
\newtheorem{lemma}{Lemma}[section]
\newtheorem{corollary}{Corollary}[section]
\newtheorem*{lemma*}{Lemma}

\theoremstyle{definition}
\newtheorem*{remark}{Remark}

\def\eps{\varepsilon}

\usepackage{amsmath,amssymb,amsfonts,amsthm,mathrsfs}

\usepackage{authblk}
\usepackage[left=2.5cm, right=3cm, top=2cm, bottom=3cm, bindingoffset=0cm, marginparwidth=3cm]{geometry}

\begin{document}

\title{New Numerical Invariants of an Unfolding \\ of a Polycycle ``Tears of the Heart''}

\date{}

\author[1]{Yu.~S.~Ilyashenko}
\author[2]{S.~Minkov}
\author[1]{I.~Shilin}

\affil[1]{HSE University, Moscow, Russia}
\affil[2]{Brook Institute of Electronic Control Machines, Moscow, Russia}

\maketitle

\begin{flushright}
	\textit{In memory of Mikhail Alexandrovich Shubin\\ a brilliant mathematician and a dear friend}
\end{flushright}

\begin{abstract}
 In this paper new numerical invariants of structurally unstable vector fields in the plane are found.
  One of the main tools is an improved asymptotics of sparkling saddle connections that occur when a separatrix loop of a hyperbolic saddle breaks. Another main tool is a new topological invariant of two arithmetic progressions, both perturbed and unperturbed, on the real line. For the pairs of the  unperturbed arithmetic progressions we give a complete topological classification.
\medskip

\noindent

\end{abstract}

%
%
%
%
\section{Introduction}

\subsection{Tears of the heart}

In \cite{IKS} an open set of structurally unstable three-parameter families of planar vector fields was found.
These families are unfoldings of a polycycle ``tears of the heart'' shown in Fig.~\ref{fig:th}.

Let $v_0$ be a vector field whose phase portrait is shown in Fig.~\ref{fig:th}.
Let L and M be hyperbolic saddles with characteristic numbers $\lambda$ and $\mu$ that satisfy $\lambda < 1, \; \lambda^2\mu > 1$.
Then the number $A = \frac {-\ln \lambda}{\ln (\lambda^2 \mu)}$ is a topological invariant of a generic three-parameter unfolding of~$v_0$, \cite{IKS}.

Here we present other invariants of the same unfolding. To describe them, we need some preparations.

\subsection{New invariants}

Let $v_0$ and $A$ be the same as before.
Let $\gamma_2$ be the whole polycycle ``tears of the heart'', and $\gamma_1 \subset \gamma_2$ be the separatrix loop. Both
polycycles are monodromic, $\gamma_1$ from the inside and $\gamma_2$ from the outside. Let $\Gamma$  be a transversal to the saddle loop, and $\Gamma_1$, $\Gamma_2$  be its interior and exterior parts with respect to $\gamma_1$.
Let $\Delta_1 \colon \Gamma_1 \to \Gamma_1$ ($\Delta_2 \colon \Gamma_2 \to \Gamma_2$) be the monodromy map (the inverse monodromy map) corresponding to $\gamma_1$ (resp., $\gamma_2$). There exist $C^1$-smooth charts $x_{1,2}$ on $\Gamma_{1,2}$ such that the monodromy maps $\Delta_j$ have the form $x_j \mapsto C_jx_j^{\nu_j}$ in the coordinates $x_1, x_2$ respectively, with $\nu_1 = \lambda$ and $\nu_2 = (\lambda^2 \mu)^{-1}$. For a map
$\Delta: x \mapsto C x^\nu$ and a small $B > 0$ we set
\begin{equation}\label{eqn:beta}
  \beta (\Delta, B) = \ln \left(\frac {\ln C}{1 - \nu} - \ln B\right).
\end{equation}
Let $B_1$ (resp., $B_2$) be the $x_1$ (resp., $x_2$) -coordinate of the intersection point of the separatrix of an interior saddle~$I$ with
$\Gamma_1$ (of a saddle $E$ with~$\Gamma_2$). Let $\beta_j = \beta (\Delta_j, B_j),\; j = 1,2$. The subsequent theorem was essentially (though not verbatim) proven in~\cite{GK:20}:

\begin{figure}
 \begin{center}
\includegraphics[width=0.5\textwidth]{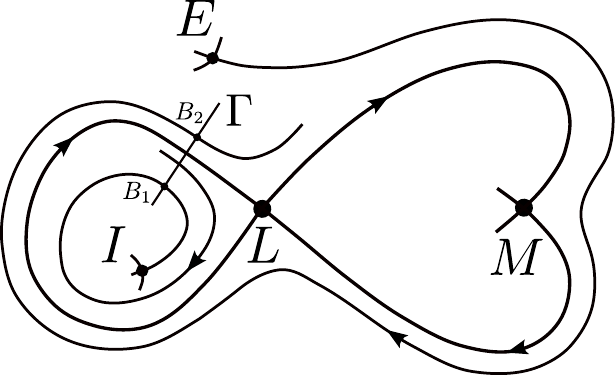}
\end{center}
\caption{ The polycycle ``tears of the heart'' with an exterior and interior saddles }\label{fig:th}
 \end{figure}

\begin{theorem}[Goncharuk and Kudryashov,~\cite{GK:20}]\label{thm:inv1}
 The equivalence class of the ratio ${\tau = \frac {\beta_1-\beta_2}{\ln \nu_2}}$,
$$
   \tau \pmod{1,A},
$$
is an invariant of the topological classification of generic three-parameter unfoldings of the vector field~$v_0$, in the case when $A$ is irrational.
\end{theorem}

\begin{remark}
The coordinate $B_j$ is not at all well defined; it may be replaced by $x_j(\Delta^k_j(B_j))$ for an arbitrary~$k$. But an easy calculation shows that this replacement changes $\frac {\beta_1-\beta_2}{\ln \lambda}$ by $As + p$ for some integers $s$ and $p$, and thus preserves the equicalence class.

We include the proof of Theorem~\ref{thm:inv1} for reader's convenience since it is short and serves as an important preliminary step in the proof of our main result.
\end{remark}

Moreover, in~\cite{GK:20} the invariance of  $\lambda$ and $\mu$  was conjectured for generic families. The positive answer to the conjecture is the main result of this paper.

\begin{theorem} \label{thm:inv2}
  There exists a set of locally topologically generic 3-parameter families of vector fields on the sphere with the following properties. Each family in the set is an unfolding of a vector field with a polycycle ``tears of the heart''. The characteristic numbers $\lambda$ and $\mu$ of the saddles $L$ and $M$, the vertices of the polycycle, are topological invariants of the families from this set. 
\end{theorem}

{\bf Definition.} The number $$\Xi=\Bigl(\frac{\ln C_1}{1-\nu_1}-\ln B_1\Bigr)^{-1}\Bigl(\frac{\ln C_2}{1-\nu_2}-\frac{\ln C_1}{1-\nu_1}\Bigr)$$
is called \emph{the relative scale coefficient} of the family.

\begin{theorem}\label{thm:inv3} For the families in Theorem~\ref{thm:inv2}, the relative scale coefficient $\Xi$ defines an invariant in the following sense:
$$
\ln |\Xi| \mod \ln\nu_1
$$
is an invariant of the topological classification.
\end{theorem}

\begin{remark} This invariant appears to be suspiciously non-symmetric, but this is not a problem. Consider another relative scale coefficient, where the roles of the two polycycles are interchanged:
$$
\Theta=\Bigl(\frac{\ln C_2}{1-\nu_2}-\ln B_2\Bigr)^{-1}\Bigl(\frac{\ln C_1}{1-\nu_1}-\frac{\ln C_2}{1-\nu_2}\Bigr).
$$
Then one has
$$
\ln |\Xi|-\ln|\Theta|=\beta_2-\beta_1.
$$

Now consider two topologically equivalent families with ratios $\tau,
\tilde\tau$ in Theorem~\ref{thm:inv1} such that
$$
\tau-\tilde\tau=As+p, \quad s,p \in \mathbb{Z}.
$$
One can prove (see Sec.~\ref{sec:exp} below)
that
$$
\ln |\Xi|-\ln|\tilde\Xi|=s\ln\nu_1.
$$
Then a direct calculation shows that
$$
\ln |\Theta|-\ln|\tilde\Theta|=p\ln\nu_2.
$$
Thus, $\ln|\Xi| \pmod{\ln\nu_1}$ and $\ln |\Theta| \pmod{\ln\nu_2}$
are representations of the same new invariant.
\end{remark}

We would like to stress the difference between the metric and the topological genericity. Consider not a functional space, but rather the real line. Diophantine irrational numbers are metrically typical (prevalent): they have the full Lebesgue measure. On the other hand, the Liouvillian numbers form a countable intersection of open and dense sets, that is, are topologically typical. Analogously to this case, our locally generic families have invariant $A$ that admits exponentially close approximations by rational numbers, but with some extra shift that depends on the other coefficients of the family, see Section 4.1, but if we considered a finitely dimensional space of these coefficients, the families of Theorems~\ref{thm:inv2} and~\ref{thm:inv3} would correspond to a subset of Lebesgue measure zero.

\subsection{Improved asymptotics of sparkling saddle connections} \label{sec:improved}

\begin{figure}\label{fig:sl}
 \begin{center}
\includegraphics[width=0.6\textwidth]{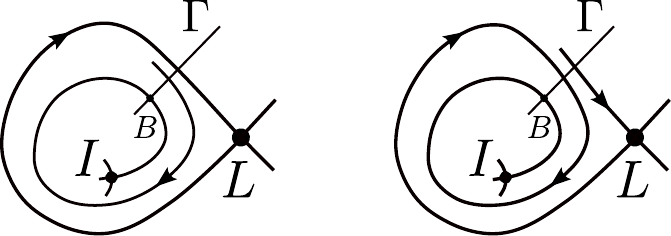}
\end{center}
\caption{ Breaking of a separatrix loop and sparkling saddle connections}
 \end{figure}

 Consider a separatrix loop of a hyperbolic saddle~$L$ with characteristic number $\lambda < 1$. Suppose that inside the loop there is another saddle~$I$ whose separatrix winds towards the separatrix loop of $L$ in the negative time.
Consider an unfolding of this vector field in a generic one-parameter family with a parameter~$\eps$. Then for a countable number of values $\eps = \eps_n \to 0$ sparkling saddle connections between $L$ and $I$ will occur: a vector field corresponding to $\eps_n$ will have a separatrix connection that makes ``$n$ turns around the vanished separatrix loop''.

In \cite{IKS} the following asymptotics for $\varepsilon_n$ was given:
\begin{equation}\label{eqn:ass0}
  \ln (-\ln \varepsilon_n) = - n \ln \lambda + O(1).
\end{equation}
 In~\cite{GK:20} an improved version of the asymptotics appeared:
\begin{equation}\label{eqn:ass2}
    \ln (-\ln \varepsilon_n) = - n \ln \lambda +\beta + o(1),
  \end{equation}
where $\beta$ is given by~\eqref{eqn:beta}.
Below we obtain yet another improvement of this formula.
\begin{lemma}\label{lem:ass1}
  \label{thm:assy} Let the monodromy map of the saddle loop above have the form $\Delta(x) = C x^\lambda$, and let the monodromy map for the perturbed saddle loop have the form $Cx^\lambda+\eps$.
  Let $B$ be the $x$-coordinate of a point where the corresponding semitransversal intersects with the separatrix of the saddle~$I$. Then
  \begin{equation}\label{eqn:ass1}
    \ln (-\ln \varepsilon_n) = - n \ln \lambda +\beta + \theta \lambda^n + o(\lambda^{n}).
  \end{equation}
  where $\beta $ is as in~\eqref{eqn:beta} and $\theta = - e^{-\beta} \frac {\ln C}{1 - \lambda}$.
\end{lemma}

A similar result holds for sparkling saddle connections of an arbitrary monodromic hyperbolic polycycle provided that a single separatrix connection is broken with nonzero speed, and we are going to apply this result to the unfolding of our ``tears of the heart'' polycycle. Namely,
consider a hyperbolic polycycle $\gamma$ that contains a separatrix loop of a saddle $L$. Suppose that this polycycle is monodromic on the outside and there is a saddle~$E$ whose separatrix winds towards $\gamma$ in the positive time.
Consider an unfolding of this vector field in a one-parameter family, parametrized by~$\eps$, that preserves all the edges of the polycycle unbroken, except for the separatrix loop. Then for a countable number of values $\varepsilon = \varepsilon_n \to 0$ sparkling saddle coonnections that involve saddles $E$ and $L$ will occur. The asymptotics of $\eps_n$ is the same: relations~\eqref{eqn:ass1} and~\eqref{eqn:ass2} hold, but with $\lambda$ replaced by the inversed product of the characteristic numbers of the saddles along the polycycle. The inversion uppears because we apply the analog of Lemma~\ref{lem:ass1} in the negative time.

These sparkling sepatatrix connections form the main obstacle for two families to be topologically equivalent.
The homeomorphism of the parameter bases that establishes the equivalence of these families must map the parameter values that correspond to connections between $L$ and $E$ to values where analogous connections occur in the second family. The same holds for the connections between $L$ and the interior saddle~$I$.

Hence, the homeomorphism of the bases must take two distinguished sequences on the parameter line into two corresponding sequences for another family. This condition is the main source of invariants. The asymptotics~\eqref{eqn:ass0} shows that in the double-logarithmic chart the sequences resemble arithmetic progressions. The number $A = \frac {-\ln \lambda}{\ln (\lambda^2 \mu)}$ is the ratio of their steps that can be interpreted as the limit relative frequency of the two sequences, which is preserved by a homeomorphism. A finer asymptotics~\eqref{eqn:ass2} yields another invariant under the assumption that $A$ is irrational. With improved asymptotics and yet another genericity assumption, we prove that the characteristic numbers $\lambda$ and $\mu$ themselves are also invariants.

\subsection{Estimates of the monodromy mapping}
The following technical lemma will play an important role in proving Theorems~\ref{thm:inv2} and~\ref{thm:inv3}.

\begin{lemma} \label{lem:est}
Consider a local one-parametric unfolding of a monodromic hyperbolic polycycle such that only one separatrix connection of the original polycycle is broken with non-zero speed with respect to the parameter $\eps$.  
Let $\nu < 1$ be the product of the characteristic numbers of the saddles encountered along the polycycle and assume that these characteristic numbers are irrational. 
Then there exists a $C^1$-smooth chart for the first return map $f_\varepsilon$ of the perturbed polycycle such that in this chart we have
\begin{itemize}
    \item[1)] $f_0(x)=Cx^\nu$;
    \item[2)] the map $f_\eps$ admits the following estimates in a domain $\{(x,\eps)\mid \varepsilon > 0, \; 1>x>\varepsilon/2\}$\footnote{We assume that for the positive values of $\eps$ the unstable separatrix involved in the broken connection goes to the ``monodromic side'' of the broken polycycle, so that the sparkling connections become possible.}:

\[
(C-k(\eps))x^\nu<f_\eps(x)<(C+k(\varepsilon))x^\nu,
\]
with $k(\varepsilon)=O(\eps^{1-\nu})$.
\end{itemize}
\end{lemma}

\subsection{Definitions}

Here we give the necessary definitions and details. The definitions and notation are mostly inherited from~\cite{IKS}. Here and below $M$ is a smooth manifold and $B \subset \mathbb{R}^k$ is a topological open ball that will serve as the base of our families.

We denote by $\mathrm{Vect}(M)$ the set of $C^\infty$-smooth vector fields on~$M$, and by $\mathrm{Vect_\star}(M)$ the subset of vector fields with a finite number of singularities. We will deal only with vector fields in $\mathrm{Vect_\star}(M)$.

\medskip

\textbf{Definition.} A \emph{family of vector fields} on a manifold $M$ with the base $B$ is a vector field $V$ on $B \times M$ tangent to the fibers $\{ \alpha \} \times M$, $\alpha \in B$. The dimension of a family is the dimension of its base.

One can also regard a family as a smooth map from the base to~$\mathrm{Vect}(M)$.

Denote by $\mathcal{V}_k(M)$ the space of $k$-parameter families of vector fields on $M$ which are $C^\infty$-smooth as vector fields on $B \times M$. We believe that we only need $C^2$-smoothness in the parameter and $C^4$-smoothness in the phase variable, but we stick to the above definition to be able to use the finitely-smooth normal form for the Dulac maps of hyperbolic saddles from~\cite{IYa}.

\medskip

\textbf{Definition.} A family of vector fields is \emph{transverse} to a Banach submanifold $T$ of $\mathrm{Vect}(M)$ provided that the corresponding map $V : B \to \mathrm{Vect}(M)$ is transverse to~$T$.

We are interested in the case where $T$ is a Banach submanifold that consists of vector fields that have a ``tears of the heart'' polycycle $\gamma$ with an exterior saddle $E$ and interior saddle $I$, as in Fig.~\ref{fig:th}, and satisfy several additional assumptions: all singularities and cycles are hyperbolic, there are no saddle connections except those of~$\gamma$, the characteristic numbers~$\lambda, \mu$ of saddles $L$ and $M$ satisfy $\lambda < 1, \; \lambda^2\mu > 1$.

We will assume that in our families the vector field $v_0$ that corresponds to the parameter value $\alpha = 0$ is distinguished  by having a ``tears of the heart'' polycycle. We say that the family is an \emph{unfolding} of the vector field $v_0$ or that the polycycle \emph{is unfolded} in this family. 
We denote by $\mathrm{T}_k^\pitchfork$ the set of $k$-parameter unfoldings of $v \in T$ that are transverse to~$T$. 

\medskip

Recall the definitions of equivalence for vector fields and families that we will be using.

\medskip

\textbf{Definition.} Two vector fields $v$ and $\tilde{v}$ on a manifold $M$ are called \emph{orbitally topologically equivalent}, if there exists a homeomorphism $M \to M$ that sends orbits of $v$ to orbits of $\tilde{v}$ and preserves their time orientation.

\medskip

\textbf{Definition.} Two families of vector fields $\{v_\alpha \mid \alpha \in B \}$, $\{\tilde{v}_{\tilde{\alpha}} \mid \tilde{\alpha} \in \tilde{B}\}$ on $M$ are called \emph{weakly topologically equivalent} if there exists a map
\begin{equation}\label{eq:weak_eqiv}
H : B \times M \to \tilde{B} \times M, \quad H(\alpha, x) = \bigl(h(\alpha), H_\alpha(x)\bigr)
\end{equation}
such that $h : B \to \tilde{B}$ is a homeomorphism, and for each $\alpha \in B$ the map 
$H_\alpha : M \to M$ is a homeomorphism that links the phase portraits of $v_\alpha$ and $\tilde{v}_{h(\alpha)}$.

Two families are \emph{strongly topologically equivalent} if there exists a \emph{homeomorphism} $H$ with above properties. Strong topological classification of families is quite rigid, whereas weak topological equivalence is too lousy, see e.g. the discussion in~\cite{GI}. The results
of~\cite{IKS} that we use are stated using the notion of \emph{moderate topological equivalence}. Hence, whenever we write about equivalent families,
we mean the moderate topological equivalence. However, we believe that the moderate equivalence can be replaced by the weak equivalence in all the
results.

Denote by $\mathrm{SPS}(v)$ the union of all singular points, limit cycles, and separatrices of a vector field $v$, and for a family $V$ denote $\mathrm{SPS}(V) = \{(\alpha, x)\colon x \in \mathrm{SPS}(v_\alpha)\}$.  

\medskip

\textbf{Definition.} We say that families $V$ and $\tilde{V}$ are \emph{moderately topologically equivalent} if there exists a linking map $H$, as in~\eqref{eq:weak_eqiv}, such that $H$ is continuous in $(\alpha, x)$ at every point of  $\mathrm{SPS}(V)$ and $H^{-1}$ is continuous at every point of~$\mathrm{SPS}(\tilde{V})$.

\medskip
Although we do not emphasize this, our results can be reformulated for \emph{local} families, i.e., for the germs of families at $\alpha = 0$.

\section{Estimates of the monodromy mapping: proof of Lemma~\ref{lem:est}}

\subsection{Preliminary remarks}

The monodromy map of a hyperbolic polycycle can be decomposed as a product of Dulac maps of saddles alternating with regular maps. We will not need notation for the corresponding transversal sections, but we assume that the coordinates on these are chosen so that the origin is the point of intersection with the polycycle, and so every map in the composition takes 0 to 0. When the polycycle is unfolded in a family of vector fields in such a way that only one separatrix connection is broken, we may assume that all regular maps, except the last one that corresponds to the broken connection, still take 0 to~0. We will separate a parameter-dependent shift $T_\eps$ from the last regular map as a compositional factor, and thus the monodromy map $f_\eps=T_\eps\circ Q_\eps$, where $Q_\eps$ will still take $0$ to $0$ for all $\eps$.
We will prove that there is a $C^1$-smooth coordinate change that turns $Q_\eps$ simply into $x \mapsto x^{\Lambda(\eps)}$, where $\Lambda(\eps)$ is the product of characteristic numbers of saddles met along the polycycle, in Lemmas~\ref{lem:composition} and~\ref{lem:normalize}. Then it will be easy to see
that the whole parameter dependent monodromy map, when written in the corresponding $C^1$-coordinates, has the form
\[f_\eps(x) = Cx^{\Lambda(\varepsilon)} + \eps (1+\psi(x^{\Lambda(\eps)},\eps)),\]
where $\psi$ is continuous and satisfies $\psi(0,0)=0$.

For $x\in (\eps/2, x_0)$ this will yield an estimate of the form
\[
(C - k(\eps))x^{\Lambda(0)} < f_\eps(x) < (C + k(\eps))x^{\Lambda(0)}
\]
which will allow us to obtain an improved asymptotic formula in the double-logarithmic coordinate for the sequence $\eps_n$ of parameter values that correspond to new saddle connections with a certain external saddle.

\subsection{Normalization of the monodromy map}

One usually decomposes a monodromy map into a product of other monodromy maps, and we also do this below when considering a product of alternating Dulac maps and regular monodromy maps. But in the statement of Lemma~\ref{lem:composition} we use a different convention.
Assume that the monodromy map $f_\eps$ is from a semitransversal $\Gamma_+$ to itself; when we apply this to the ``tears of the heart'' polycycle, the role of $\Gamma_+$ will be played by $\Gamma_1$ or $\Gamma_2$. We view $f_\eps$ as a composition $T_\varepsilon \circ Q_\varepsilon$, where $T_\varepsilon, Q_\varepsilon$ are also maps from $\Gamma_+$ to itself.
Thus, when we make a coordinate change later that brings $Q_\varepsilon$ to its normal form $x \mapsto x^{\Lambda(\varepsilon)}$, the same coordinate change is to be applied to $T_\varepsilon$, which is quite convenient.

\begin{lemma} \label{lem:composition} Let $f_\eps$ be a family of monodromy maps along a perturbed monodromic hyperbolic polycycle in the special one-parameter local family of vector fields where only one saddle connection is broken. Assume that the family is $C^1$-smooth in $\eps$ and all vector fields are $C^\infty$-smooth in~$x$ and that at $\eps = 0$ the characteristic numbers of the saddles of the polycycle are irrational. 
Then for any smooth coordinate on the transversal, the map $f_\eps$ admits a decomposition 
$$
f_\eps(x)=T(Q(x, \eps),\eps) = T_\varepsilon \circ Q_\varepsilon(x)
$$    
where $T$ is a $C^1$-smooth $\varepsilon$-dependent shift and $Q_\varepsilon(\cdot) = Q(\cdot, \varepsilon)$ always takes 0 to 0 and has the form
$$
Q(x,\eps)=C(\eps)x^{\Lambda(\eps)}+O(x^{\Lambda(\eps)+c}),
$$
where $c > 0$ and $O$-big is uniform in~$\eps$. Here, $\Lambda(\eps)$ is the product of characteristic numbers of the saddles along the polycycle.
\end{lemma}

\begin{proof}
Define $Q_\varepsilon$ as $Q_\varepsilon(x) = f_\varepsilon(x) - f_\varepsilon(0)$, then $Q_\varepsilon(0) \equiv 0$ and we have $f_\varepsilon = T_\varepsilon \circ Q_\varepsilon$ with $T_\varepsilon(x) = x + f_\varepsilon(0)$.

Let $N$ be the number of saddles in our hyperbolic polycycle. The characteristic numbers of these saddles are irrational for $\eps=0$, so due to \cite{IYa} we have a finitely smooth normal form 
\[
\Delta_j(x,\eps):=x^{\lambda_j(\eps)}
\]
for the Dulac maps of the saddles. These normal forms are obtained in some parameter-dependent coordinates on the semitransversals to the stable and the unstable separatrix of each saddle, and we write the regular monodromy maps between the saddles using the same coordinates. The regular maps along non-broken separatrices have the form:
\[
\varphi_j(y,\eps)=a_j(\eps)\,y + b_j(\eps)\,y^2 + r_j(y,\eps),
\]
with continuous $a_j>0, b_j, r_j$ and $r_j$ being $o(y^2)$ uniformly in $\eps$ as $y \to 0$. 
Since we have separated the shift $T_\varepsilon$ from $Q_\varepsilon$, we can assume that the same is true for the last regular map in the decomposition of $Q_\varepsilon$.

For notational convenience, we usually omit $\eps$ in the notation for $\Delta_j, \varphi_j$ and write, e.g., $f = T \circ Q$ instead of $f(x,\eps) = T(Q(x,\eps),\eps)$.
With this convention, the map $Q$ can be written as
$$
Q=\varphi_N\circ \Delta_N \circ \dots \circ \varphi_1\circ\Delta_1, 
$$
assuming that we write $Q$ in the normalizing coordinate (on $\Gamma_+$) of the first Dulac map.

The lemma is proved by induction in the number of Dulac maps.

Define recursively $Q_0(x,\eps):=x$ and, for $k=1,\dots,N$,
\[
Q_k:=\varphi_k\circ \Delta_k \circ Q_{k-1}.
\]

For $k = 1,\dots, N,$ let us denote
\[
\Lambda_k(\eps):=\prod_{i=1}^k \lambda_i(\eps), \qquad
s_k(\eps)=\min_{1\le i\le k}\Lambda_i(\eps).
\]

We prove by induction on $k$ that for some smooth $C_k(\eps)$ bounded away from zero we have
\begin{equation}\label{eq:inductive-form}
Q_k(x)=C_k(\eps)\,x^{\Lambda_k(\eps)} + R_k(x,\eps),
\qquad |R_k(x,\eps)|\le K_k\,x^{\Lambda_k(\eps)+s_k(\eps)},
\end{equation}
for all sufficiently small $\eps$ and all $x\in[0,x_k]$, where $x_k>0$ and $K_k>0$ are independent of~$\eps$.

\emph{Base of induction: $k=1$.}
Using the uniform Taylor expansion of $\varphi_1$, we write
\[
Q_1(x)=\varphi_1\bigl(x^{\lambda_1(\eps)}\bigr)
= a_1(\eps)\,x^{\Lambda_1(\eps)} + b_1(\eps)\,x^{2\Lambda_1(\eps)}
+ r_1\bigl(x^{\lambda_1(\eps)},\;\eps\bigr).
\]

The remainder is uniformly small, so there exists $x_1>0$ such that
$|r_1(x^{\lambda_1(\eps)},\eps)|\le x^{2\Lambda_1(\eps)}$ for $x\in[0,x_1]$ and all small~$\eps$. Hence
\eqref{eq:inductive-form} holds with $C_1=a_1$, $s_1=\Lambda_1(\eps)$, $K_1 = \sup |b(\eps)|+1$.

\emph{Inductive step.}
Assume \eqref{eq:inductive-form} holds with $k$ replaced by~$k-1$. Write
\[
Q_{k-1}(x,\eps)=A(x,\eps)+E(x,\eps),
\]
with 
\[A:=C_{k-1}(\eps)\,x^{\Lambda_{k-1}(\eps)},\quad
|E|\le K_{k-1}\,x^{\Lambda_{k-1}(\eps)+s_{k-1}(\eps)}.
\]
Set $\eta:=E/A$. By assumption, $\inf C_{k-1}(\eps) >0$, so we have $|\eta(x,\eps)|\le \tilde K\,x^{s_{k-1}(\eps)}$ for $x\in[0,\tilde x]$. Then, using the expansion $(1+\eta)^{\lambda}=1+\lambda\eta+O(\eta^2)$, we have
\[
\Delta_k(Q_{k-1})=\bigl(A\cdot(1+\eta)\bigr)^{\lambda_k(\eps)}
= A^{\lambda_k(\eps)}\Bigl(1+\lambda_k(\eps)\eta+O(\eta^2)\Bigr).
\]
Thus, for some $K'>0$ and $x\in[0,\tilde x]$,
\[
\Delta_k(Q_{k-1}(x))
= C_{k-1}(\eps)^{\lambda_k(\eps)}\,x^{\Lambda_k(\eps)}
\;+\; \widetilde E(x,\eps),\qquad
|\widetilde E(x,\eps)|\le K' x^{\Lambda_k(\eps)+s_{k-1}(\eps)}.
\]
Applying the uniform second-order expansion of $\varphi_k$, we get
\[
Q_k=\varphi_k (\Delta_k\circ Q_{k-1})
= a_k(\eps)\,(\Delta_k\circ Q_{k-1})
+ b_k(\eps)\,(\Delta_k\circ Q_{k-1})^2 + r_k\bigl(\Delta_k\circ Q_{k-1},\eps\bigr).
\]
Substituting the previous expansion of $Q_{k-1}$, using the uniform control of $r_k$, and shrinking $\tilde x$ if necessary, we obtain
\[
Q_k(x)
= a_k(\eps)\,C_{k-1}(\eps)^{\lambda_k(\eps)}\,x^{\Lambda_k(\eps)}
+ E_1(x,\eps)+E_2(x,\eps),
\]
where
\[
|E_1|\le K_1\,x^{\Lambda_k(\eps)+s_{k-1}(\eps)},\quad
|E_2|\le K_2\,x^{2\Lambda_k(\eps)}.
\]
We set $C_k(\eps)=a_k(\eps)\,C_{k-1}(\eps)^{\lambda_k(\eps)}$ and observe that it is smooth and bounded away from zero. Since $s_k(\eps):=\min\bigl(s_{k-1}(\eps),\,\Lambda_k(\eps)\bigr)$, we have
\[
Q_k(x)
= C_k(\eps)\,x^{\Lambda_k(\eps)} + R_k(x,\eps)\quad
\text{ with }\quad
|R_k(x,\eps)|\le K_k\,x^{\Lambda_k(\eps)+s_k(\eps)}.
\]
Taking $k=N$ and
$c<\min(1,s_N(0)/2)$, we get the required statement. Recall that we were working with $Q$ written in the normalizing chart of~$\Delta_1$. We omit the straightforward check that, since $c<\Lambda(\eps)$ and $c<1$, the map $Q$ will have the same form $Q(x, \eps)=C(\eps)x^{\Lambda(\eps)}+O(x^{\Lambda(\eps)+c})$ in any smooth $\eps$-dependent coordinate, but with a different~$C(\eps)$.

\end{proof}


\begin{lemma}\label{lem:normalize}
Let $Q_\eps$ be a local family of uniformly bounded $C^4$-maps defined for small $x>0$ and satisfying
\[
Q_\eps(x)=C(\eps)\,x^{\lambda(\eps)} + O\!\big(x^{\lambda(\eps)+c}\big),\qquad x\searrow0,
\]
uniformly in $\eps$, where $\lambda(\cdot),C(\cdot)$ are $C^2$-smooth positive functions and $\lambda(0) < 1$. Then $\{Q_\eps\}$ is $C^1$-conjugate near $x = 0$ to a model family
$$x\mapsto x^{\lambda(\eps)}.$$
More precisely: there exists a family of parameter-dependent diffeomorphisms $\Phi_\eps$, $C^1$-smooth in $(x, \eps)$, with $\Phi_\eps(0)=0$ and $\Phi_\eps'(0)>0$, such that
\[
(\Phi_\eps^{-1}\circ Q_\eps\circ \Phi_\eps) (x)=x^{\lambda(\eps)}
\]
for all $x \in [0,x_0]$ with $x_0$ sufficiently small.
\end{lemma}

\begin{proof}
\begin{enumerate}
 \item A local family is a germ of a family at $\eps = 0$. Since $C(\eps)$ is bounded and separated from 0 for small $\eps$ and $\lambda(\eps)$ is separated from 1, we can make an $\eps$-dependent rescaling of $x$
and remove the coefficient~$C(\eps)$. Thus, it suffices to treat the case where
$Q_\eps(x)=x^{\lambda(\eps)}+O(x^{\lambda(\eps)+c})$.

 \item Consider the chart
\[
y=-\frac{1}{\ln x}\qquad(x=e^{-1/y},\; y\searrow 0 \text{ as } x\searrow 0).
\]
In this chart the map  $x\mapsto Q_\eps(x)$ is transformed into a map
\[
y\mapsto G_\eps(y)\;=\; -\frac{1}{\ln\big(Q_\eps(e^{-1/y})\big)} .
\]
We have $Q_\eps(e^{-1/y})=e^{-\lambda(\eps)/y}+O\big(e^{-(\lambda(\eps)+c)/y}\big)$ and hence

\[
G_\eps(y)=\frac{1}{\lambda(\eps)}y + R(y,\eps),
\qquad R(y,\eps)=o(y^2).\
\]
In particular, $G'_\eps(0)=\frac{1}{\lambda(\eps)}.$

\item The map $(y,\eps) \mapsto G_\eps(\cdot)$ can be viewed as a smooth family (with parameter~$\eps$) of germs at $y=0$ whose linear parts are $y\mapsto y/\lambda(\eps)$.
 Using $C^4$-regularity in $y$ and $C^2$-regularity in $\eps$, we apply the parametric version of the Sternberg linearization theorem proven by G.~R.~Sell \cite[Theorem~9]{Sell:85} and obtain a family $\Psi$ of local diffeomorphisms
$
\Psi_\eps(y)
$
which is $C^2$ in $(y,\eps)$, with $\Psi_\eps(0)=0$ and $\Psi'_\eps(0)\ne0$, and conjugates $G_\eps(\cdot)$ to its linear part:
\[
(\Psi_\eps^{-1}\circ G_\eps\circ \Psi_\eps)(y)=\frac{y}{\lambda(\eps)} .
\]

Note that one can replace $\Psi_\eps(y)$ with $\Psi_\eps(y)/ \Psi'_\eps(0)$ to guaranty  $\Psi_\eps'(0)=1$ for all~$\eps$.

\item The conjugacy $\Psi_\eps$ in the $y$-chart transforms into a conjugacy $\Phi_\eps$ in the original $x$-variable by the relation $x=e^{-1/y}$. The only point yet to check is the regularity of $\Phi$ at $x=0$; this is verified below in Proposition~\ref{prop:C1-chart}.
\end{enumerate}
\end{proof}

\begin{proposition}\label{prop:C1-chart}
Let $g \colon (\mathbb R^{+},0) \to (\mathbb R^{+},0)$ be a germ of a $C^2$-smooth function such that 
$$g(y)=y+R(y) \qquad R(y)=O(y^2),$$
and let $h(x)=-\frac{1}{\ln x}$. Then $H=h^{-1}\circ g\circ h$ is $C^1$-smooth.

Moreover, if $g$ depends on a parameter $\eps$ and $g''(x)$ is $C^1$-smooth in~$\eps$, then $H'$ also is $C^1$-smooth w.r.t.~$\eps$.
\end{proposition}

\begin{proof}
Obviously, $h^{-1}(y)=e^{-1/y}$. 
By the chain rule,
$$
H'(x)= \exp(-1/(g(h(x))) \cdot \frac{1}{g(h(x))^{2}} \cdot (1+R'(h(x)))\cdot\frac{1}{x \ln^2 x}.
$$

Thus, for $x>0$ the map $H$ is $C^1$, and $H'$ also $C^1$-smoothly depends on $\eps$ provided that $R$ smoothly depends on~$\eps$.

When $x \searrow 0$, one has $R'(h(x))\to 0$ and  $g(h(x))^2\ln^2x\to 1$. Hence
$$
\lim_{x\to+0} H'(x)=\lim_{x\to+0}  \exp(-1/(g(h(x)))/x=\lim_{x\to+0}  \exp(-1/(g(h(x)))-\ln(x)).
$$
Put $y=h(x) = -1/\ln x$. Then
$$
\frac{-1}{g(h(x))}-\ln(x)=\frac{-1}{y+R(y)}-\frac{-1}{y}=\frac{R(y)}{y(y+R(y))}\to g''(0)/2, \quad y\to 0.
$$

Therefore
$$
\lim_{x\to0} H'(x)=e^{g''(0)/2}.
$$
When $g''$ smoothly depends on $\eps$, so does $H'$, and we are done.
\end{proof}

We apply Proposition~\ref{prop:C1-chart} to $g = \Psi$ to check that $\Psi$ is smooth at $x = 0$.

\subsection{The estimates of the monodromy map}\label{sec:ineq_for_P}
In Lemma~\ref{lem:composition} above, we separated a parameter-dependent shift $T_\varepsilon$ from the monodromy map $f_\eps$. The important property of this factor is that it becomes the identity at $\eps = 0$: $T_0(x) = x + f_0(0) = x$ for all $x$ near zero. Unlike the property of being a shift, this property is preserved under an arbitrary coordinate change. Recall that the monodromy map along the perturbed polycycle decomposes as
\[f(x, \eps) = T_\eps \circ Q_\eps(x) = Q_\eps(x) + f_\eps(0),\]
and by Lemma~\ref{lem:normalize}, in a certain $C^1$-coordinate the map $Q_\eps$ rewrites as $x \mapsto  x^{\Lambda(\eps)}$. In the same $\eps$-dependent coordinate,\footnote{Recall that we view $T_\eps$ and $Q_\eps$ as maps from the semitransversal $\Gamma_+$ to itself.} $T_\eps$ rewrites as $\varphi_\eps(\cdot) = \varphi(\cdot, \eps)$, which is in general only $C^1$-smooth, but retains the property $\varphi_0(u) \equiv u$.

\begin{proof}[Proof of Lemma~\ref{lem:est}]
We work in the normalizing coordinate for $Q_\eps$.
Denote
\[
   \Lambda:=\Lambda(0) < 1,\qquad K:=\partial_{\eps}\varphi(0,0).
\]
Write the first-order expansion of $\varphi$ in $\eps$, uniform in $u$ near $0$:
\begin{equation}\label{eq:fexp}
\varphi(u,\eps)=u+\eps\,\partial_\eps \varphi(u,0)+\eps\,\eta(u,\eps), \quad \eta(u,\eps) = o(1), \quad \eps \searrow 0,
\end{equation}
and substitute $u=Q(x,\eps) = x^{\Lambda(\eps)}$:
\[
\varphi(x^{\Lambda(\eps)},\eps)
=x^{\Lambda(\eps)}+\eps\,\partial_\eps \varphi(x^{\Lambda(\eps)},0)+\eps\,\eta(x^{\Lambda(\eps)},\eps).
\]

Since $\Lambda(0)>0$, we have $x^{\Lambda(\varepsilon)}\to 0$ as $x\searrow 0$, and hence
$\partial_\eps \varphi(x^{\Lambda(\eps)},0)\to\partial_\eps \varphi(0,0)=K$, and the convergence is uniform for $\eps$ near 0. Thus,
$$
\varphi(x^{\Lambda(\eps)},\eps)=x^{\Lambda(\eps)}+\eps(K+\tilde\eta(x^{\Lambda(\eps)},\eps)).
$$

By changing the scale, we can get rid of the constant $K$ in the $K\eps$ term, but at the cost of obtaining a constant in the $x^{\Lambda(\varepsilon)}$ term \footnote{We assume that the separatrix connection of the original polycycle is broken with non-zero speed with respect to the parameter~$\eps$, and therefore $K\neq0$.}:
\begin{equation}\label{eq:monodromy_expansion}
f(x, \eps) = Cx^{\Lambda(\varepsilon)} + \eps (1+\psi(x^{\Lambda(\eps)},\eps)),
\end{equation}
where $\psi$ is continuous and $\psi(0,0) = 0$.
In the rest of the paper, whenever the monodromy map of a perturbed polycycle is considered, we will work in a coordinate where it has this form, and we will refer to such a coordinate as a parameter-dependent \emph{canonical chart}. The chart itself is not uniquely defined by our conditions, but the constant~$C$ is, and we refer to it as one of \emph{the constants of the family}.

Consider
the domain $\{x\colon \eps < x < x_0 < 1\}$, where $x_0$ is small.
By~\eqref{eq:monodromy_expansion}, we have
$$
\left| \frac{f(x, \eps)}{Cx^\Lambda}-1 \right |= \left|x^{\Lambda(\eps)-\Lambda}-1+\frac{\eps}{x^\Lambda}(1+\psi(x^{\Lambda(\eps)},\eps))\right|.
$$

For the term $x^{\Lambda(\eps)-\Lambda}$ we can write
\[x^{\Lambda(\eps)-\Lambda} =
x^{O(\eps)} = \bigl( x^\eps \bigr)^{O(1)}.
\]
We have $1 + \eps\ln\eps < e^{\eps\ln\eps} = \eps^\eps < 1$. The number $x^\eps$ is between $\eps^\eps$ and $1$, and therefore is inside the segment $[1+\eps \ln\eps, \;1]$. Hence,
\[
x^{\Lambda(\eps)-\Lambda} - 1 = 
\bigl( x^\eps \bigr)^{O(1)} -1 = \left(1 + O\bigl(\eps \ln\eps\bigr)\right)^{O(1)} - 1 =  O\bigl(\eps \ln\eps\bigr), \quad \eps \to +0,
\]
which yields
$$
\left| \frac{f(x,\eps)}{Cx^\Lambda}-1 \right | = O(\eps\ln\eps)+O(\eps^{1-\Lambda}).
$$

For sufficiently small $\eps > 0$, we have $\eps|\ln\eps| < \eps^{1-\Lambda}$,
so there exists $k>0$ such that for small $\eps$ and all $x \in (\eps, x_0)$ one has
\begin{equation}\label{eq:g_and_h}
g_\eps(x) := (C-k\eps^{1-\Lambda} )x^\Lambda<f_\eps(x)<(C+k\eps^{1-\Lambda})x^\Lambda =: h_\eps(x).
\end{equation}

The proof of Lemma \ref{lem:est} is finished.
\end{proof}

\begin{remark}
    The same argument works in the domain $\{x \colon  \eps/2 < x < x_0\}$, only the constant $k$ will be different.
\end{remark}

\section{Preliminary description of the base homeomorphism}\label{sec:3}

\subsection{The connection equation: the general case and the model case}

Let $B(\eps)$ be the coordinate of the point where the separatrix of the external saddle first intersects the semitransversal $\Gamma_+$ where our monodromy map $f_\eps$ is defined. Since it is a stable separatrix, we may as well call it the point of the \emph{last} intersection. The condition for the appearance of a connection between this saddle and saddle $L$ can be written as $f^{n+1}_{\eps_n}(0)=B(\eps_n)$; we will refer to it as the connection equation.
It was proved in~\cite{IKS} that the solutions $\eps_n$ exist and form a monotone sequence that tends to zero.

First, let us consider the model example where
\[f_\eps(x) = Cx^{\Lambda(\varepsilon)} + \eps.\]
We will deal with the term $\eps\psi(x^{\Lambda(\eps)}, \eps)$ of the actual map $f_\eps$ later. In the model case, the connection equation can be rewritten as
$f^{n}_{\eps_n}(\eps_n)=B(\eps_n)$.

Note that the coordinate $B(\eps)$ depends smoothly on $\eps$, so we have $B(\eps) = B + O(\eps)$. With the sequence $(\eps_n)_n$ fixed and the functions $g_\eps, \; h_\eps$ defined in~\eqref{eq:g_and_h}, consider the equations
\[
g_{\eps_n}^n(\delta_n) = B(\eps_n), \qquad h_{\eps_n}^n(\mu_n) = B(\eps_n).
\]
For all small $\eps$ one can write a common lower estimate for the position of the attracting fixed point for the functions $g_\eps, h_\eps$. If $B(\eps_n)$ lies to the left of this point --- which is the case if $B$ is small\footnote{This can be achieved by replacing $B$ with its image under some iterate of $f_0$, see \cite{IKS}.} --- then both equations are solvable. Moreover, the inequality
\[
\mu_n < \eps_n < \delta_n
\]
holds.

Indeed, suppose that $\delta_n \le \eps_n$. Then, by the monotonicity of $g_{\eps_n}$ and the inequality $g_{\eps_n} < f_{\eps_n}$, we obtain
\[
B(\eps_n) = g_{\eps_n}^n(\delta_n) \leq g_{\eps_n}^n(\eps_n) < f_{\eps_n}^n(\eps_n) = B(\eps_n).
\]
The resulting contradiction shows that $\eps_n < \delta_n$.
The second inequality is checked analogously.

After applying the double logarithm and taking signs into account, we obtain
\begin{equation}\label{eq:e_n_ineq}
\ln(-\ln \delta_n)<\ln(-\ln \eps_n)<\ln(-\ln \mu_n).
\end{equation}

\subsection{Asymptotics of the sparkling saddle connections: the model case}\label{sect:taylor}
We will prove the generalization of Lemma \ref{lem:ass1} that deals with an arbitrary monodromic polycycle perturbed in a one-parameter family in such a way that only one separatrix connection is broken. First, we give the proof for the model case where the monodromy map has the form $f_\eps(x) = Cx^{\Lambda(\eps)} + \eps$. Then we discuss the general case\footnote{Recall that in order to have this form for the perturbed monodromy map we assume that the characteristic numbers of the saddles in the polycycle are irrational and that for their product we have $\Lambda(0) < 1$.} of $f_\eps(x) = Cx^{\Lambda(\varepsilon)} + \eps (1+\psi(x^{\Lambda(\eps)},\eps))$.

\begin{lemma}[Generalized Lemma~\ref{lem:ass1}]
\label{lem:taylor}
In the case of a monodromic polycycle (with irrational characteristic numbers that have product $\Lambda < 1$) perturbed in a one-parameter family where only one connection is broken, the sequence $(\eps_n)$ of parameter values that correspond to separatrix connections with an external saddle has the form
    $$
    \ln(-\ln\eps_n)=-n\ln\Lambda+\beta+\theta\Lambda^n+R_n,
    $$
    where
    $$
    \beta=\ln\left(\frac{1}{1-\Lambda}\ln C-\ln B\right),
    $$
    $$
    \theta=-\left(\frac{\ln C}{1-\Lambda}-\ln B\right)^{-1}\frac{\ln C}{1-\Lambda},
    $$
and $R_n=o(\Lambda^{n})$; $B$ and $C$ were described above.
\end{lemma}

\begin{proof}[Proof for the model case.]
For the function $y(x)=Cx^\Lambda$ its $n$-th iterate has the form
$$
y^n(x)=C^{1+\Lambda+\dots+\Lambda^{n-1}}x^{\Lambda^n}=C^{\frac{1-\Lambda^n}{1-\Lambda}}x^{\Lambda^n}.
$$

Set $C(\eps) = C-k\eps^{1-\Lambda}$ for $\eps > 0$. By the same reasoning, the equation for $\delta_n$, $B(\eps_n) = g_{\eps_n}^n(\delta_n)$, takes the form
$$
B(\eps_n)=C(\eps_n)^{\frac{1-\Lambda^n}{1-\Lambda}}\delta_n^{\Lambda^n}.
$$
Solving this for $\delta_n$ and taking the logarithm twice\footnote{Note that we may choose $B(\eps)$ sufficiently small so that the quantity
$
-\ln B(\eps_n)+\frac{1-\Lambda^n}{1-\Lambda}\ln C(\eps_n)
$
is positive and bounded away from zero for all $n$, and hence its logarithm is well defined.}, we obtain
$$
\ln(-\ln \delta_n)=-n\ln \Lambda+\ln\left(-\ln B(\eps_n)+\frac{1-\Lambda^n}{1-\Lambda}\ln{C(\eps_n)}\right).
$$
An analogous identity holds for $\mu_n$. Therefore, \eqref{eq:e_n_ineq} becomes
\[
-n\ln\Lambda+\ln\!\left(-\ln B(\eps_n)+\frac{1-\Lambda^n}{1-\Lambda}\ln C(\eps_n)\right)
<\ln(-\ln \eps_n),
\]
\[
\ln(-\ln \eps_n)
<-n\ln\Lambda+\ln\!\left(-\ln B(\eps_n)+\frac{1-\Lambda^n}{1-\Lambda}\ln \hat C(\eps_n)\right),
\]
where $\hat C(\eps_n)=C+k\eps_n^{1-\Lambda}$.

In particular, these inequalities imply that $\iota^{-\Lambda^{-n}}<\eps_n<I^{-\Lambda^{-n}}$ for some $\iota > I > 1$.

 Consider the term
\[
\ln\!\left(-\ln B(\eps_n)+\frac{1-\Lambda^n}{1-\Lambda}\ln C(\eps_n)\right).
\]
This term is of the form $\ln(a+b)$ where
\[
a = \frac{1}{1-\Lambda}\ln C(\eps_n)-\ln B(\eps_n),\qquad
b = -\frac{\Lambda^n}{1-\Lambda}\ln C(\eps_n).
\]
Since $b\to 0$ as $n\to+\infty$ while $a$ is positive and  bounded away from zero, we can write
\[
\ln(a+b)=\ln a + \ln\bigl(1+\tfrac{b}{a}\bigr)=\ln a + \frac{b}{a} + O(b^2),\qquad b\to0.
\]
In our case, this yields
\[
\begin{aligned}
\ln\Bigl(\frac{1-\Lambda^n}{1-\Lambda}\ln C(\eps_n)-\ln B(\eps_n)\Bigr)
&=\ln\Bigl(\frac{\ln C(\eps_n)}{1-\Lambda}-\ln B(\eps_n)\Bigr)\\
&\qquad{}-\Bigl(\frac{\ln C(\eps_n)}{1-\Lambda}-\ln B(\eps_n)\Bigr)^{-1}\frac{\Lambda^n}{1-\Lambda}\ln C(\eps_n)
+O(\Lambda^{2n}).
\end{aligned}
\]
Since for any smooth function $y$ we have $y(x+\delta)=y(x)+O(\delta)$ as $\delta\to0$, and since the argument of the logarithm in the above expression is bounded away from zero under the appropriate choice of $B(\eps)$, we obtain
\begin{equation}\label{eq:taylor_expansion_in_3.2}
\begin{aligned}
\ln\Bigl(\frac{1-\Lambda^n}{1-\Lambda}\ln C(\eps_n)-\ln B(\eps_n)\Bigr)
&=\ln\Bigl(\frac{\ln C}{1-\Lambda}-\ln B\Bigr)
-\Bigl(\frac{\ln C}{1-\Lambda}-\ln B\Bigr)^{-1}\frac{\Lambda^n}{1-\Lambda}\ln C\\
&\qquad{}+O(\Lambda^{2n})+O(\eps_n^{1-\Lambda}).
\end{aligned}
\end{equation}
Now, since $\eps_n<I^{-\Lambda^{-n}}$ with $I>1$, it follows that
\[
O(\eps_n^{1-\Lambda})=O\!\bigl( I^{-(1-\Lambda)\Lambda^{-n}}\bigr)=o(\Lambda^{2n}),
\]
the latter relation being valid because
\[
\frac{I^{-(1-\Lambda)\Lambda^{-n}}}{\Lambda^{2n}}
=\left(\frac{I^{-(1-\Lambda)}}{(\Lambda^{2n})^{\Lambda^n}}\right)^{\Lambda^{-n}}
<(1-\tilde\delta)^{\Lambda^{-n}}=o(1),\qquad n\to+\infty,
\]
for some small positive $\tilde\delta$.

The same argument can be applied to~\eqref{eq:taylor_expansion_in_3.2} with $C(\eps_n)$ replaced with $\hat{C}(\eps_n)$. This shows that both $\ln(-\ln \delta_n)$ and $\ln(-\ln \mu_n)$ admit the same asymptotic expansion
\[
-n\ln\Lambda+\beta+\theta\Lambda^n+O(\Lambda^{2n}), \quad n \to +\infty.
\]
By~\eqref{eq:e_n_ineq}, the same applies to $\ln(-\ln\eps_n)$.
This proves Lemma~\ref{lem:taylor} in the model case.
\end{proof}

\subsection{Asymptotics of the sparkling saddle connections: general case}\label{sec:gen_case}
\begin{proof}[Proof of Lemma~\ref{lem:taylor} in the general case.]
Recall that in a special $C^1$-chart on the semitransversal we have
\[f_\eps(x) =
Cx^{\Lambda(\varepsilon)} + \eps(1+\psi(x^{\Lambda(\eps)},\eps)), \quad \psi(0, \eps) = o(1), \quad \eps \searrow 0,\]
and the connection equation has a less convenient form
\[f^n_{\eps_n}(\eps_n(1 + \psi(0,  \eps_n))) = B(\eps_n).\]
For the sequence $\tilde \eps_n=\eps_n(1 + \psi(0,  \eps_n)) = \eps_n\cdot (1+o(1))$, the same argument as above works\footnote{Note that it is possible that $\tilde{\eps}_n < \eps_n$, but we can be sure that $\tilde{\eps}_n > \eps_n/2$ for large~$n$. This is why we had to discuss earlier that our estimates $g_\eps < f_\eps < h_\eps$ also work in the domain $\{x \ge \eps/2\}$.}, and we obtain the same expansion as in the model case:
$$
\ln(-\ln \tilde \eps_n)=-n\ln\Lambda+\beta+\theta\Lambda^n+R_n.
$$

On the other hand,
$$
\ln(-\ln \tilde \eps_n)=\ln(-\ln \eps_n(1+o(1)))=\ln(-\ln \eps_n+o(1))=
\ln(-\ln \eps_n)+o(1/\ln{\varepsilon_n}).
$$
But $o(1/\ln{\varepsilon_n})=o(\Lambda^n)$ as $n \to +\infty$, and we are done:
\[
\ln(-\ln \eps_n)=-n\ln\Lambda+\beta+\theta\Lambda^n+o(\Lambda^n), \quad n \to +\infty.
\]
\end{proof}
Note that the lemma can be applied to the sequence of saddle connections between the saddles $I$ and $L$ with $\Lambda = \lambda$ and for connections between $E$ and $L$ with $\Lambda=(\lambda^2\mu)^{-1}$.

\subsection{Topology of two pairs of arithmetic progressions on the line}\label{sec:seq}

\begin{proposition}\label{prop:arith}
For irrational numbers $A,\tilde A$, the equivalence
$$
\forall m, n \in \mathbb{Z}\colon \quad A n+\tau\leq m\Leftrightarrow \tilde A n+\tilde t\leq m
$$
implies $A=\tilde A$ and $\tau=\tilde t$.
\end{proposition}

\begin{proof} Consider the case when $A>0$. Let
$$
n(m)=\max\{n \mid A n+\tau\leq m \},
$$    
$$
\tilde n(m)=\max\{n \mid \tilde An+\tilde t\leq m \}.
$$

Then we have $n(m)=\tilde n(m)$ and $n(m)=\big[\frac{m-\tau}{A} \big]$. Hence,
$$
1=\lim_{n\to\infty} \frac{n(m)}{\tilde n(m)}=\frac{\tilde A}{A}.
$$

Now, assume that $\tau<\tilde t$. As $A$ is irrational, there exist $n,m \in \mathbb Z$ such that
$$
m-A n\in(\tau,\tilde t).
$$
Thus $A n+\tau<m<A n+\tilde t,$
which contradicts the assumptions of the proposition.

The cases of $\tau>\tilde t$ and $A<0$ are similar. Note that in the case $A>0$ it suffice to have the equivalence for large positive $n, m$ instead of all integers.
\end{proof}

{\bf Definition.} For a pair of arithmetic progressions $(\alpha n+\beta,\;\gamma m+\delta)$, we define their \emph{relative density $A$} and \emph{normalized difference} of the free terms $\tau$ by the following expressions:
$$
A=\frac{\alpha}{\gamma}, \qquad\tau=\frac{\beta-\delta}{\gamma}.
$$

{\bf Definition.} We shall say that two pairs of sequences on the line are \emph{topologically equivalent at $+\infty$} if there exists a homeomorphism of the line, defined in a neighborhood of $+\infty$, which maps the first sequence of the first pair to the first sequence of the second pair and the second sequence of the first pair to the second sequence of the second pair.

\begin{lemma}\label{lem:ar_pr_eq} Two pairs of increasing arithmetic progressions with irrantional relative densities are topologically equivalent iff their relative densities coincide, and the normalized differences of the free terms coincide modulo the Abelian group generated by their relative density and 1. In formulas:
$$
(\alpha n+\beta,\;\gamma m+\delta)\sim (\tilde \alpha n+\tilde\beta,\;\tilde\gamma m+\tilde\delta) \Leftrightarrow
$$
$$
\frac{\alpha}{\gamma}=\frac{\tilde \alpha}{\tilde \gamma}; \qquad \frac{\beta-\delta}{\gamma}=\frac{\tilde \beta-\tilde \delta}{\tilde  \gamma}\mod\left(1,\frac{\alpha}{\gamma}\right).
$$
\end{lemma}
\begin{proof} Apply the inverse mappings of $x\mapsto \gamma x+\delta$ and $x\mapsto \tilde \gamma x+\tilde\delta$ to the first sequences in the pairs and reduce the lemma to the proposition. In more detail, denote $x_n=\alpha \gamma^{-1}n + (\beta-\delta)\gamma^{-1}$, $\tilde x_n=\tilde\alpha \tilde\gamma^{-1}n + (\tilde\beta-\tilde\delta)\tilde\gamma^{-1}$.

Let $h$ be a homeomorphism that establishes the equivalence of the new two pairs:
$$
(x_n, m)\sim(\tilde x_n,m).
$$
Then for some integers $p$ and $s$ 
$$
h(x_n)=\tilde x_{n+s}, \qquad h(m)=m+p.
$$   
Denote $A=\alpha \gamma^{-1}$, $\tau=(\beta-\delta)\gamma^{-1}$, $\tilde A=\tilde\alpha \tilde\gamma^{-1}$, $\tilde \tau=(\tilde\beta-\tilde\delta)\tilde\gamma^{-1}$, and $\tilde t=(\tilde\beta-\tilde\delta)\tilde\gamma^{-1}+As-p$.

Then
$$
An+\tau\leq m \Leftrightarrow  \tilde A n+\tilde t\leq m,
$$
and, by Proposition~\ref{prop:arith}, we have $A= \tilde A$, $\tau=\tilde t$ and thus $\tau=\tilde\tau\mod (1,A)$.

The inverse statement of the lemma is obvious. 
\end{proof}

\subsection{Topology of two pairs of slightly perturbed arithmetic progressions}

\textbf{Definition.}
 Let us call a sequence of the form
$$
x_n=A n+\tau+o(1), \qquad n\to +\infty,
$$
\emph{a slightly perturbed arithmetic progression}.

\begin{remark} By Lemma \ref{lem:ass1}, the sequence $\ln(-\ln\eps_n)$ is a slightly perturbed arithmetic progression.
\end{remark}

For slightly perturbed arithmetic progressions, the \emph{only if} part of the previous lemma holds true.

\begin{lemma} If two pairs of slightly perturbed increasing arithmetic progressions are topologically equivalent, then their relative densities coinside, and the normalized differencies of the free terms coinside modulo the Abelian group generated by their common relative density and 1.
\end{lemma}

\begin{proof} The analog of Proposition~\ref{prop:arith} holds true for slightly perturbed arithmetic progressions (one only needs to replace the condition $m-A n\in(\tau,\tilde t)$ by a slightly more restrictive  $m-A n\in(0.9 \tau+0.1\tilde t,0.9\tilde t+0.1\tau)$ to obtain the full proof). The rest of the proof is the same as for Lemma~\ref{lem:ar_pr_eq}.
\end{proof}

\begin{corollary} \label{cor:shifts} Consider two pairs of equivalent slightly perturbed increasing arithmetic progressions
$$(x_n,y_m)\sim (\tilde x_n,\tilde y_m)$$
with irrational density~$A$. If their density and normalized differencies of the free terms satisfy
$$
\tau-\tilde\tau=As-p
$$
with $s, p \in \mathbb Z$, 
then for the corresponding homeomorphism $h$ we must have $h(x_n)=\tilde x_{n+s}$, $h(y_m)=y_{m+p}$.
\end{corollary}

\begin{proof} We can check that $p$ and $s$ are uniquely defined.
Indeed, if another pair $p_2, s_2$ satisfied the same relation:
\[
\tau-\tilde\tau=As_2-p_2,
\]
then
\[
As-p=As_2-p_2
\]
so $(s-s_2)A=p-p_2$, which would imply $A$ rational, contradicting the irrationality of~$A$.

The rest is obvious from the proof of the lemmas.
\end{proof}

\section{The new invariants}
\subsection{Sketch of the proof}
{\bf Definition.} We call a sequence $(x_n)_{n \in \mathbb{N}}$ \emph{an exponentially perturbed arithmetic progression} with \emph{base}~$\lambda < 1$ if this sequence admits an expansion
$$
x_n=A n+\tau+\xi\lambda^n+ o(\lambda^n), \qquad n\to \infty,
$$
with $A, \tau, \xi \in \mathbb{R}$.

Note that the base of an exponentially perturbed arithmetic progression is correctly (uniquely) defined.

{\bf Definition.} For a pair of exponentially perturbed arithmetic progressions
\begin{equation}\label{eq:ar_prog_pair_form}
(\alpha n + \beta +\xi \nu_1^n + o(\nu_1^n), \; \gamma n + \delta + \psi \nu_2^n + o(\nu_2^n)),
\end{equation}
we define their \emph {relative scale coefficient} by the following expression:
$$
\psi\nu_2^\tau-\xi
$$
with $\tau=\frac{\beta-\delta}{\gamma}$.

In this section we will deal with exponentially perturbed arithmetic progressions that originate from the topologically distinguished one-parameter subfamilies of the three-parameter unfoldings of the ``tears of the heart'' polycycle. More precisely, we will define a set of so called shift-exp-Liouvillian 3-parameter families of vector fields with a ``tears of the heart'' polycycle at parameter value zero. In a one-parameter subfamily where the ``heart'' is preserved, we have two sequences $(i_n), (e_n)$ of parameter values that correspond to saddle connections between the external saddles~$E$ or~$I$ and the saddle~$L$. When written in the double-logarithmic chart, these sequences become exponentially perturbed arithmetic progressions.

\begin{remark} We will show in Lemma~\ref{lem:exp} that the  relative scale coefficient of the perturbed arithmetic progressions $(i_n), (e_n)$, defined in terms of the coefficients and bases of the progressions, is equal to the  relative scale coefficient of the family defined in the introduction in terms of the characteristic numbers and coefficients $B_j, C_j$ of the original vector field with the ``tears of the heart'' polycycle.
\end{remark}

For a pair of our shift-exp-Liouvillian families, the equivalence of the pairs $((i_n),(e_n))$ and $((\tilde i_n), (\tilde e_n))$ will imply that the bases of these exponentially perturbed arithmetic progressions coincide:
\[\lambda = \tilde\lambda, \quad (\lambda^2\mu)^{-1} = (\tilde{\lambda}^2\tilde{\mu})^{-1}.\] 

\begin{remark}
The set of shift-exp-Liouvillian families is distinguished by a condition on the parameters of the family, or rather of the vector field with the ``tears of the heart'' polycycle. The condition resembles the one in the definition of Liouvillian numbers, and, similarly to the case of numbers, it defines a topologically generic subset in the space of coefficients, but of zero Lebesgue measure.
\end{remark}

The proof of the main result can now be split into the following three lemmas.
\begin{lemma*} For a topologically generic 3-parameter family that unfolds the ``tears of the heart'' polycycle, the sequences $(i_n), (e_n)$ are exponentially perturbed arithmetic progressions.
\end{lemma*}
This is just a reformulation of Lemma~\ref{lem:ass1} in the general case; we proved it in Section~\ref{sec:gen_case}.
\begin{lemma} \label{lem:exp} Let $((i_n), (e_n))$ and $((\tilde i_n),(\tilde e_n))$ be the pairs of exponentially perturbed arithmetic progressions generated by two shift-exp-Liouvillian families. If these pairs are topologically equivalent at infinity,
then the bases of the sequences $(i_n)$ and $(\tilde i_n)$ coincide, the bases of $(e_n)$ and $(\tilde e_n)$ coincide, and the logarithms of their relative scale coefficients coincide modulo the logarithm of the base of~$(i_n)$ and~$(\tilde{i}_n)$.
\end{lemma}

\begin{lemma} \label{lem:generic} The set of shift-exp-Liouvillian families is residual in the class $\mathrm{T}_3^\pitchfork$ of all 3-parameter unfoldings of the ``tears of the heart'' polycycle.
\end{lemma}

We will prove Lemma~\ref{lem:exp} in Section~\ref{sec:exp}, and Lemma~\ref{lem:generic} will be proven in Section~\ref{sec:generic}.
Now we can prove Theorems~\ref{thm:inv2} and~\ref{thm:inv3} modulo the three lemmas above.

\begin{proof}[Proof of Theorems~\ref{thm:inv2} and~\ref{thm:inv3}] Let two shift-exp-Liouvillian families be moderately topologically equivalent. Then the pairs of sequences $((i_n),(e_n))$ and $((\tilde i_n), (\tilde e_n))$ generated by these families are topologically equivalent. Then Lemma~\ref{lem:exp} yields $\lambda=\tilde\lambda$. Since the parameter $A = \frac {-\ln \lambda}{\ln (\lambda^2 \mu)}$ must be the same for the equivalent families, the equality $\mu=\tilde\mu$ follows.
This completes the proof of Theorem~\ref{thm:inv2}.
Likewise, by Lemma~\ref{lem:exp} the logarithms of the relative scale coefficients coincide modulo~$\ln\lambda$, which yields Theorem~\ref{thm:inv3}.
\end{proof}

\subsection{Shift-exp-Liouvillian three-parameter families}

Recall that $\mathrm{T}_3^\pitchfork$ is the set of three-parameter unfoldings of vector fields with ``tears of the heart'' polycycles and that the number
$$\Xi=\Bigl(\frac{\ln C_1}{1-\nu_1}-\ln B_1\Bigr)^{-1}\Bigl(\frac{\ln C_2}{1-\nu_2}-\frac{\ln C_1}{1-\nu_1}\Bigr)$$
is called the relative scale coefficient of the family. We assume that the corresponding sequences $(i_n), (e_n)$ have the form
\[
i_n = \alpha n + \beta +\xi \nu_1^n + o(\nu_1^n), \quad e_n = \gamma n + \delta + \psi \nu_2^n + o(\nu_2^n), \quad n\to+\infty.
\]

{\bf Definition.}  A three-parameter family from $\mathrm{T}_3^\pitchfork$ is called \emph{shift-exp-Liouvillian} if it satisfies the following conditions.
\begin{itemize}
\item The relative scale coefficient $\Xi$ is not zero.
\item The ratio $A = \frac {-\ln \lambda}{\ln (\lambda^2 \mu)} = \frac{\alpha}{\gamma}$ is irrational. Moreover, for every positive rational number $q\in[0.5,1)\cup(1,2]$ there exist arbitrarily large $m,n \in \mathbb{N}$ such that \footnote{Here and below we denote by $[a, b]$ the segment between the points $a$ and~$b$ without regard to whether $a\le b$ or vice versa.}
$$
A-\frac{m}{n}\in\frac{u+[\,q^2\Xi\lambda^n,\;q \Xi\lambda^n]}{\gamma n},
$$
where $u=\delta-\beta$.
\end{itemize} 

\subsection{Shift-exp-Liouvillian families have two additional invariants}  \label{sec:exp}

We can prove Lemma~\ref{lem:exp} now.

\begin{proof}[Proof of Lemma~\ref{lem:exp}]
    
Consider two shift-exp-Liouvillian families. In the double logarithmic coordinate, the sequences of values of $\varepsilon$ that correspond to the saddle connections that we consider have the form
\[
 i_n = \alpha n + \beta+\xi \lambda^{n}+R_1(n), \quad
 e_m = \gamma m +\delta+\psi \Lambda^{m} +R_2(m),
\]
\[
\tilde i_n = \tilde\alpha n +\tilde\beta+\tilde \xi \tilde \lambda ^{n}+R_3(n), \quad
\tilde e_m = \tilde\gamma m +\tilde\delta+\tilde\psi \tilde \Lambda^{m} +R_4(m),
\]
where $\Lambda=(\lambda^2\mu)^{-1}$, $R_1(n)=o(\lambda^{n})$, $R_2(m)=o(\Lambda^{m})$, $R_3(n)=o(\tilde \lambda^{n})$, $R_4(m)=o(\tilde \Lambda^{m})$. The explicit formulas for $\psi,\xi$ were obtained above in Lemma~\ref{lem:taylor}.

According to Corollary \ref{cor:shifts} there are shifts $p,s$ such that the homeomorphism $h$ that establishes the equivalence of the pairs satisfies
\[
h(i_n) = \tilde{i}_{n+s}, \quad
h(e_m) = \tilde{e}_{m+p}.
\]
Since $h$ preserves the order, for every relevant $n,m$ we must have
\[i_n < e_m\]
if and only if
\[\tilde{i}_{n+s} < \tilde{e}_{m+p}.\]
Substituting the explicit expressions of $i_n, e_m, \tilde{i}_{n+s}, \tilde{e}_{m+p}$ and rearranging the terms, we get an equivalent statement:
\[
\gamma\Bigl(A n+\frac{\beta-\delta}{\gamma}-m\Bigr) < \bigl(\psi \Lambda^{m} - \xi \lambda^{n}\bigr) + R_1(n)-R_2(m)
\]
must hold if and only if
\[
\tilde\gamma\Bigl(A(n+s)+\frac{\tilde\beta-\tilde\delta}{\tilde\gamma}-m-p\Bigr) < \bigl(\tilde\psi \tilde\Lambda^{m+p} - \tilde\xi \tilde\lambda^{n+s}\bigr) + R_3(n+s)-R_4(m+p).
\]
By the choice of $p,s$ in Corollary \ref{cor:shifts} we have
\[
A n+\frac{\beta-\delta}{\gamma}-m = A(n+s)+\frac{\tilde\beta-\tilde\delta}{\tilde\gamma}-m-p,
\]
so the equivalence of the two inequalities above implies that
\[
An+\frac{\beta-\delta}{\gamma}-m \notin \Bigl(\tfrac{1}{\gamma}\bigl(\psi \Lambda^{m}-\xi \lambda^{n}+R_1(n)-R_2(m)\bigr),\ \tfrac{1}{\tilde\gamma}\bigl(\tilde\psi \tilde\Lambda^{m+p}-\tilde\xi \tilde\lambda^{n+s}+R_3(n+s)-R_4(m+p)\bigr)\Bigr),
\]
or, equivalently,
\begin{equation}\label{eq:notin}
A-\frac{m}{n}\notin \frac{\delta-\beta+(Q_1(n,m),Q_2(n,m))}{\gamma n},
\end{equation}
where
\[
Q_1(n, m)=\psi \Lambda^{m}-\xi \lambda^{n}+R_1(n)-R_2(m),
\]
\[
Q_2(n, m)=\frac{\gamma}{\tilde \gamma}\bigl(\tilde\psi \tilde\Lambda^{m+p}-\tilde\xi \tilde\lambda^{n+s}+R_3(n+s)-R_4(m+p)\bigr).
\]

The endpoints of the segment $[Q_1(n,m),Q_2(n,m)]$ exponentially converge to zero as $n,m\to+\infty$; hence there exists $R>1$ such that for all sufficiently large $n,m$ we have
\[
[Q_1(n,m),Q_2(n,m)]\subset[-R^{-\min(n,m)},\;R^{-\min(n,m)}].
\]
If $(n,m)$ is such that $|A-m/n|> \mathrm{const}$ and $n$ is large, then non-inclusion~\eqref{eq:notin} is automatically satisfied.
It is therefore natural to consider only those pairs for which
\[
\left|A-\frac{m}{n}+\frac{\tau}{n}\right|<\frac{1}{n^2},
\qquad\text{where }\tau=\frac{\beta-\delta}{\gamma},
\]
since only for such pairs can condition~\eqref{eq:notin} possibly fail for large $n$. Let us call such pairs \emph{good}. For large $n$ there is at most one $m=m(n)$ making the pair good.

\begin{proposition}
There exist constants $\Xi_1, \Xi_2 \ne 0$ such that for good pairs $(n,\;m(n))$ we have
\[
Q_1(n,m(n))\sim \Xi_1\lambda^{n},\qquad Q_2(n,m(n))\sim \Xi_2\tilde\lambda^{n}, \quad n\to+\infty.
\]
\end{proposition}
\begin{proof}
Write $m=An+\tau+O(1/n)$. Put $R(n,m)=R_1(n)-R_2(m)$. Since $\Lambda^A = \lambda$, we have
\[
\begin{aligned}
Q_1(n,m)&=\psi \Lambda^{m}-\xi \lambda^{n}+R(n,m)\\
&=\psi \Lambda^{An+\tau+O(1/n)}-\xi\lambda^{n}+R(n,m)\\
&=\psi \lambda^n\Lambda^{\tau}(1+o(1))-\xi\lambda^{n}+o(\lambda^n)\\
&=\bigl(\psi\Lambda^{\tau}-\xi\bigr)\lambda^n+o(\lambda^n).
\end{aligned}
\]
In the notation of Section~\ref{sect:taylor} we have $\nu_1=\lambda$, $\nu_2=\Lambda$ and
\[
\xi=\frac{\ln C_1}{\ln C_1-(1-\nu_1)\ln B_1},\qquad
\psi=\frac{\ln C_2}{\ln C_2-(1-\nu_2)\ln B_2},
\]
\[
\tau=\frac{\ln\!\bigl(\tfrac{\ln C_1}{1-\nu_1}-\ln B_1\bigr)-\ln\!\bigl(\tfrac{\ln C_2}{1-\nu_2}-\ln B_2\bigr)}{-\ln\nu_2}.
\]
Then we have
\[
\nu_2^{\tau}=\frac{\tfrac{\ln C_2}{1-\nu_2}-\ln B_2}{\tfrac{\ln C_1}{1-\nu_1}-\ln B_1},
\]
and hence
\[
\psi\nu_2^{\tau}-\xi=\Bigl(\frac{\ln C_1}{1-\nu_1}-\ln B_1\Bigr)^{-1}\Bigl(\frac{\ln C_2}{1-\nu_2}-\frac{\ln C_1}{1-\nu_1}\Bigr).
\]
Since we are dealing with shift-exp-Liouvillian families, we have $\Xi_1=\psi\Lambda^{\tau}-\xi = \Xi\neq0$, and the asymptotic for $Q_1$ follows.

For $Q_2(n,m(n))$ the same computation (with tilded quantities and the shifts $p,s$) yields
\[
Q_2(n,m)=\frac{\gamma}{\tilde\gamma}\bigl(\tilde\psi\tilde\Lambda^{m+p}-\tilde\xi\tilde\lambda^{n+s}+\tilde R(n,m)\bigr)
=\Xi_2\,\tilde\lambda^{n}+o(\tilde\lambda^{n}),
\]
with
\[
\Xi_2=\frac{\gamma}{\tilde\gamma}\bigl(\tilde\psi\tilde\Lambda^{\tilde\tau}-\tilde\xi\bigr)\tilde\lambda^{s}\neq0.
\]
This proves the proposition.
\begin{remark} Note that $\Xi_2=\frac{\gamma}{\tilde\gamma}\bigl(\tilde\psi\tilde\Lambda^{\tilde\tau}-\tilde\xi\bigr)\tilde\lambda^{s}$ differs from the relative scale coefficient of the second family $\tilde \Xi = \bigl(\tilde\psi\tilde\Lambda^{\tilde\tau}-\tilde\xi\bigr)$  only by a factor $\tilde\lambda^{s}{\gamma}/{\tilde\gamma}$. 
\end{remark}
\end{proof}

Thus, for the good pairs $(n,m)$ we have 
\[
[Q_1(n,m),\;Q_2(n,m)] = [\Xi_1\lambda^n+o(\lambda^n),\ \Xi_2\tilde\lambda^n+o(\tilde\lambda^n)], \quad n \to +\infty.
\]

Assume that $\lambda\neq\tilde\lambda$ or $\Xi_1\neq\Xi_2$. For some rational number  $q\in[0.5,1)\cup(1,2]$ and all sufficiently large $n$ we have \footnote{For the case of positive $\Xi_i$ one can put $q=0.5$ provided $\tilde \lambda<\lambda$ and $q=2$ for  $\tilde \lambda>\lambda$; for $\tilde \lambda=\lambda$ and  $\Xi_1\neq \Xi_2$ one can take $q\in((\Xi_2/\Xi_1)^{1/4},(\Xi_2/\Xi_1)^{1/3})$. The other cases are now obvious.}
\[
[q^2\Xi_1\lambda^{n},q\Xi_1\lambda^n]\subset [Q_1(n,m),\;Q_2(n,m)] = [\Xi_1\lambda^n+o(\lambda^n),\ \Xi_2\tilde\lambda^n+o(\tilde\lambda^n)].
\]

Since the family is shift-exp-Liouvillian, there are infinitely many $n$ for which the pair $(n,m(n))$ violates~\eqref{eq:notin}.
If two pairs of sequences were equivalent at infinity, this would be impossible. Therefore, for equivalent pairs (and families) we must have
\[
\lambda=\tilde\lambda, \qquad \Xi_1= \Xi_2.
\]
Note that $\lambda=\tilde\lambda$ implies $\Lambda = \tilde\Lambda$ and $\gamma=\tilde \gamma$, and thus $\Xi_2=\tilde \lambda^{s}\tilde \Xi = \lambda^{s}\tilde \Xi$. Hence $\ln \Xi=\ln\tilde \Xi+s\ln \lambda$.
This is exactly the statement of Lemma~\ref{lem:exp}.
\end{proof}

\subsection{Shift-exp-Liouvillian families are generic} 
\label{sec:generic}
Let us prove Lemma~\ref{lem:generic}.

\begin{proof}[Proof of Lemma~\ref{lem:generic}]
The condition $\Xi\neq 0$ can be rewritten as
\[
\ln C_1 \cdot \bigl(1-(\lambda^2\mu)^{-1}\bigr)\neq \ln C_2 \cdot (1-\lambda).
\]
Let $U_{0,0}$ be the set of vector fields in $\mathrm{T}_3^\pitchfork$ for which, in canonical coordinates on the transversal, this condition holds. The openness and density of $U_{0,0}$ are obvious. 

Define the subsets of $Q^+_{N,q}(\gamma,u,\Xi,\lambda) \subset\mathbb R$ by
\[
Q^+_{N,q}(\gamma,u,\Xi,\lambda)=\bigcup_{n,m>N}\left(\frac{m\gamma+u+q^2 \Xi \lambda^{\,n}}{\gamma n},\ \frac{m\gamma+u+q\Xi\lambda^{\,n}}{\gamma n}\right),
\]
with $\Xi\neq0,\lambda>0$.

Each $Q^+_{N,q}(\gamma,u,\Xi,\lambda)$ is open as a union of open intervals. We show it is dense in $\mathbb{R}$. Take an arbitrary interval $[c-r,c+r]$. Choose a point of the form $\frac{m\gamma+u}{\gamma N_r}$ inside $(c-r/2,c+r/2)$ with denominator $N_r$ so large that $\max(q^2 \Xi \lambda^{\,N_r},q \Xi \lambda^{\,N_r})<\gamma r/2$ and $N_r>N$. Then
\[
\left(\frac{m\gamma+u+q^2 \Xi \lambda^{\,N_r}}{\gamma n},\ \frac{m\gamma+u+q\Xi\lambda^{\,N_r}}{\gamma n}\right)\subset[c-r,c+r],
\]
so every interval contains points of $Q^+_{N,q}(\gamma,u,\Xi,\lambda)$. 

Let $U_{N,q}$ be the set of three-parameter families in $\mathrm{T}_3^\pitchfork$ for which the constant $A$ belongs to $Q^+_{N,q}(\gamma,u,\Xi,\lambda)$. Each $U_{N,q}$ is open and dense in $\mathrm{T}_3^\pitchfork$.

Openness is immediate: mall perturbations of a family change its parameters only slightly,\footnote{Note that the canonical chart depends
continuously on the family of vector fields, see Theorem~A in P.~Berger and
B.~Reinke, arXiv:2210.05256, 2022.} so an inclusion
\[
A\in \left(\frac{m\gamma+u+q^2 \Xi \lambda^{\,n}}{\gamma n},\ \frac{m\gamma+u+q\Xi\lambda^{\,n}}{\gamma n}\right),
\]
persists under small perturbations.

To prove density, note that by an arbitrarily small perturbation we can change $\mu$ so that 1) $\lambda$ remains the same, and 2) the constant $A$ of the perturbed family falls into $ Q^+_{N,q}(\gamma,u,\Xi,\lambda)$, where $\gamma,u,\Xi,\lambda$ are the constants of the original family. Then, by a second small perturbation that adjusts $C_1,C_2,B_1$ (without touching the saddle characteristic numbers) we may arrange that the constants $u/\gamma,\;\Xi/\gamma$ of the perturbed family coincide with the original $u/\gamma,\;\Xi/\gamma$. After these two arbitrarily small perturbations, the resulting family lies in $U_{N,q}$. Hence $U_{N,q}$ is dense.

Therefore the intersection $\bigcap_{N,q} U_{N,q}\cap U_{0,0}$ is the desired residual set that consists of shift-exp-Liouvillian families.
\end{proof}

\section*{Funding}
The study was implemented in the framework of the Basic Research Program at the HSE University in 2025, project №~075-00648-25-00 ``Symmetry. Information. Chaos''.

\section*{Conflicts of interest}

The authors of this work declare that they have no conflicts of interest.

\medskip
\noindent
{\large \bf Yu. S. Ilyashenko,\\}
HSE University, Moscow, Russia\\
E-mail: \texttt{\tt yulijs@gmail.com}

\medskip
\noindent
{\large \bf Stanislav Minkov,\\}
Brook Institute of Electronic Control Machines, Moscow, Russia\\
E-mail: \texttt{stanislav.minkov@yandex.ru}

\medskip
\noindent
{\large \bf Ivan Shilin,\\}
HSE University, Moscow, Russia\\
E-mail: \texttt{i.s.shilin@yandex.ru}


\begin{thebibliography}{99}

\bibitem{IKS}
Yu.~Ilyashenko, Yu.~Kudryashov, and I.~Schurov, 
``Global bifurcations on the two-sphere: a new perspective,''
\textit{Inventiones Mathematicae}, 
vol.~213, no.~2, 2018, pp.~461--506.

\bibitem{IYa}
Yu.~Ilyashenko and S.~Yakovenko, 
``Finitely-smooth normal forms of local families of diffeomorphisms and vector fields,''
\textit{Russian Mathematical Surveys}, 
vol.~46, no.~1, 1991, pp.~1--43.

\bibitem{GI}
N.~B.~Goncharuk and Yu.~S.~Ilyashenko, 
``Various Equivalence Relations in Global Bifurcation Theory,''
\textit{Proceedings of the Steklov Institute of Mathematics}, 
vol.~310, 2020, pp.~78--97.

\bibitem{GK}
N.~Goncharuk and Yu.~Kudryashov, 
``Families of vector fields with many numerical invariants,''
\textit{Discrete and Continuous Dynamical Systems}, 
vol.~42, no.~1, 2022, pp.~239--259.

\bibitem{GK:20}
N.~Goncharuk and Yu.~Kudryashov, 
``Bifurcations of the polycycle ‘tears of the heart’: multiple numerical invariants,''
\textit{Moscow Mathematical Journal}, 
vol.~20, no.~2, 2020, pp.~323--341.

\bibitem{Sell:85}
G.~R.~Sell, 
``Smooth linearization near a fixed point,''
\textit{American Journal of Mathematics}, 
vol.~107, no.~5, 1985, pp.~1035--1091.

\end{thebibliography}
\end{document}